\documentclass[12pt]{amsart}

\usepackage{amsmath}
\usepackage{amsthm}
\usepackage{amsfonts}
\usepackage{comment}
\usepackage{amssymb,amsrefs}
\usepackage{hyperref}
\usepackage{mathtools}
\usepackage{bm}
\usepackage[margin=1in]{geometry}

\usepackage{mathrsfs}
\usepackage{enumerate}

\usepackage{todonotes, atveryend}

\newtheorem{theorem}{Theorem}

\newtheorem{lemma}[theorem]{Lemma}

\theoremstyle{definition}

\newtheorem*{theorem*}{Theorem}
\newtheorem*{proposition*}{Proposition}
\newtheorem*{lemma*}{Lemma}

\theoremstyle{remark}
\newtheorem*{remark}{Remark}

\newcommand{\FF}{\mathbb F}
\newcommand{\QQ}{\mathbb Q}

\newcommand{\ZZ}{\mathbb Z}

\makeatletter
\newcommand*{\defeq}{\mathrel{\rlap{%
                     \raisebox{0.3ex}{$\m@th\cdot$}}%
                     \raisebox{-0.3ex}{$\m@th\cdot$}}%
                     =}
\makeatother

\makeatletter
\def\@tempa#1{\@xp\@tempb\meaning#1\@nil#1}
\def\@tempb#1>#2#3 #4\@nil#5{%
  \@xp\ifx\csname#3\endcsname\mathaccent
    \@tempc#4?"7777\@nil#5%
  \else
    \PackageWarningNoLine{amsmath}{%
    Unable to redefine math accent \string#5}%
  \fi
}
\def\@tempc#1"#2#3#4#5#6\@nil#7{%
  \chardef\@tempd="#3\relax\set@mathaccent\@tempd{#7}{#2}{#4#5}}

\@tempa\widetilde
\makeatother

\begin{document}

\title[Permutations that destroy arithmetic progressions]{Permutations that destroy arithmetic progressions \\ in Elementary $p$-Groups}

\date{\today}

\author[Noam D. Elkies]{Noam D. Elkies}

\author[Ashvin A. Swaminathan]{Ashvin A. Swaminathan}


\begin{abstract}
Given an abelian group $G$, it is natural to ask whether there
exists a permutation $\pi$ of $G$ that ``destroys'' all nontrivial
3-term arithmetic progressions (APs), in the sense that $\pi(b)
- \pi(a) \neq \pi(c) - \pi(b)$ for every ordered triple $(a,b,c) \in
G^3$ satisfying $b-a = c-b \neq 0$. This question was resolved
for infinite groups $G$ by Hegarty, who showed that there exists
an AP-destroying permutation of $G$ if and only if $G/\Omega_2(G)$
has the same cardinality as $G$, where $\Omega_2(G)$ denotes the
subgroup of all elements in $G$ whose order divides $2$. In the
case when $G$ is finite, however, only partial results have been
obtained thus far. Hegarty has conjectured that an AP-destroying
permutation of $G$ exists if $G = \ZZ/n\ZZ$ for all $n \neq 2,3,5,7$,
and together with Martinsson, he has proven the conjecture for
all $n > 1.4 \times 10^{14}$. In this paper, we show that if $p$
is a prime and $k$ is a positive integer, then there is an AP-destroying
permutation of the elementary $p$-group $(\ZZ/p\ZZ)^k$ if and only
if $p$ is odd and $(p,k) \not\in \{(3,1),(5,1), (7,1)\}$.
\end{abstract}
\vspace*{-0.1in}
\maketitle

\section{Introduction}\label{sec:intro}

Let $G$ be an abelian group, and let $\pi : G \to G$\/ be any permutation.
Following the terminology of Hegarty and Martinsson (see~\cite{hegand}),
we say that $\pi$ \emph{destroys} all nonconstant arithmetic progressions
(henceforth, APs) in $G$ if there is no ordered triple $(a,b,c)
\in G^3$ such that $b-a = c-b \neq 0$ and $\pi(b) - \pi(a) = \pi(c)
- \pi(b)$ (i.e., $(a,b,c)$ and $(\pi(a),\pi(b),\pi(c))$ are never both APs;
note that this condition holds for $\pi$ if an only if it holds for $\pi^{-1}$).
It is natural to seek a complete classification of abelian groups~$G$\/
 that have such \mbox{an AP-destroying permutation.}

For $G$\/ infinite, it was shown by Hegarty in~\cite{hegori}
that there exists an AP-destroying permutation of $G$ if and only
if $G/\Omega_2(G)$ has the same cardinality as $G$, where $\Omega_2(G)$
denotes the subgroup of all elements in $G$ whose order divides~$2$.
On the other hand, in the case when $G$ is finite, such a classification
has not yet been obtained. Hegarty conjectured in~\cite{hegori}
that there exists an AP-destroying permutation of $\ZZ/n\ZZ$ for
all $n \neq 2,3,5,7$ (in these four cases, one readily checks
that there is no AP-destroying permutation of $\ZZ/n\ZZ$).
It was shown by Hegarty and Martinsson in~\cite{hegand} that there exists
an AP-destroying permutation of $\ZZ/n\ZZ$ for all
$n \geq n_0 = (9 \cdot 11 \cdot 16 \cdot 17 \cdot 19 \cdot 23)^2
\approx 1.4 \times 10^{14}$.
Moreover, the following lemma, proved by Hegarty in~\cite{hegori},
shows that one can find AP-destroying permutations of larger groups
given AP-destroying permutations of smaller groups:
\begin{lemma}[Hegarty]\label{lem:quot}
Let $G$ be an abelian group and $H \subset G$ a subgroup. If there
exists an AP-destroying permutation of $H$ and an AP-destroying
permutation of $G/H$, then there exists an AP-destroying permutation
of $G$.
\end{lemma}
It follows from Lemma~\ref{lem:quot} that the set of all finite
abelian groups $G$ that have AP-destroying permutations is closed
under taking direct sums, and that the set of all positive integers
$n$ for which $\ZZ/n\ZZ$ has an AP-destroying permutation is closed
under multiplication. This last implication motivates a closer
study of the case when $n = p$ is a prime, and in this regard,
it was shown in~\cite{hegand} that there is an AP-destroying permutation
of $\ZZ/p\ZZ$ for all primes $p > 3$ such that $p \equiv 3 \pmod
8$. In this paper, we prove a result that includes Hegarty's conjecture
for $\ZZ/p\ZZ$ where $p > 7$ is any prime. Our main theorem is
stated as follows:
\begin{theorem}\label{thm:main}
Let $p$ be a prime and $k$ be a positive integer. Then there is
an AP-destroying permutation of $(\ZZ/p\ZZ)^k$ if and only if $p$
is odd and $(p,k) \not\in \{(3,1),(5,1), (7,1)\}$.
\end{theorem}

\begin{remark}
Although one can use Lemma~\ref{lem:quot} to show that there is
an AP-destroying permutation of $(\ZZ/p\ZZ)^k$ for all $k > 3$
given the existence of such a permutation for $(\ZZ/p\ZZ)^k$ where
$k = 1$ or where $k \in \{2,3\}$, our approach yields
such a permutation directly for all elementary $p$-groups
of odd order greater than~$7$.  Using this result, together with
Lemma~\ref{lem:quot} and the result of~\cite{hegand},
we then see that to prove Hegarty's conjecture for finite cyclic groups,
it suffices to find AP-destroying permutations of $\ZZ/n\ZZ$
for all $n$ such that $n \in \{2p,3p,5p,7p : p \text{ prime}\}$
and $n < n_0$.
\end{remark}

The proof of Theorem~\ref{thm:main} occupies the remainder of this paper.
We first construct a permutation $f$\/ that destroys all but $O(1)$ APs
in $(\ZZ/p\ZZ)^k$.  We then show that if $p^k$ is large enough, say
$p^k > n_1$, then a small modification of~$f$\/ destroys all APs.
This $n_1$, unlike the bound $n_0$ of~\cite{hegand}, is small enough
that we can deal with the remaining cases $q^k \leq n_1$
by exhibiting an AP-destroying permutation in each case.
This concludes the proof.

\section{Proof of Theorem~\ref{thm:main}}\label{sec:proof}

Let $p$\/ be any prime, let $k$ be any positive integer,
and let $G$\/ denote the elementary $p$-group $(\ZZ/p\ZZ)^k$. If $\pi$ is a permutation of $G$, then
for any $a,r \in G$ with $r \neq 0$, the permutation $\pi$ destroys the AP
$(a-r, a, a+r)$ if and only if $\pi$ destroys the reversed AP $(a+r, a, a-r)$.
Moreover, if $p=3$, then all six permutations of the AP $(a-r, a, a+r)$
are APs, and if $\pi$ destroys one of them then $\pi$ destroys them all.
Thus, in the remainder of our proof, we will somewhat loosely use the notation
``$(a-r, a, a+r)$'' to refer to both the AP $(a-r, a, a+r)$ and
the reversed AP $(a+r, a, a-r)$, and when $p = 3$, the notation
``$(a-r, a, a+r)$'' will refer to any permutation of this AP.

If $p=2$ then no permutation of $G$\/ destroys any AP,
so we need only consider the case when $p$ is odd.
We identify $(\ZZ/p\ZZ)^k$ with the additive group of
the finite field $\FF_q$ of order $q = p^k$.
We shall construct an AP-destroying permutation of $(\ZZ/p\ZZ)^k$
as a permutation of $\FF_q$, by applying small modifications
to the fixed permutation $f : \FF_q \to \FF_q$ defined by
\begin{equation}
\label{eqn:almostwreck}
f(x) \defeq \begin{cases}
   1 & \text{if $x = 0$}, \\ 0 & \text{if $x = 1$},
   \\
   \frac{1}{x} & \text{else.}
   \end{cases}
\end{equation}
The next lemma shows that
$f$\/ is indeed very close to being an AP-destroying permutation:
\begin{lemma}\label{lem:almostwreck}
The permutation $f$ destroys all APs in $\FF_q$ other than $(-1,0,1)$
when $p = 3$ and $(0,\frac32,3)$, $(\frac13, \frac23, 1)$ when
$p > 3$.
\end{lemma}
\begin{proof}
Let $(a-r, a, a+r)$ be an AP in $\FF_q$ such that $\{a-r, a, a+r\}
\cap \{0,1\} = \varnothing$. Then $f$ sends the AP $(a-r, a, a+r)$
to $\left(\frac{1}{a-r}, \frac{1}{a}, \frac{1}{a+r}\right)$, which
is an AP when
$$
\frac{2}{a} = \frac{1}{a-r} + \frac{1}{a+r}
\ \Longrightarrow \
2(a^2 - r^2) = 2a^2
\ \Longrightarrow \
2r^2 = 0,
$$
but this cannot hold since $p > 2$.
Thus, all APs disjoint from $\{0,1\}$ are destroyed by~$f$.
The remaining cases are handled as follows:
\begin{enumerate}
\item First, consider APs of the form $(-r, 0, r)$. If $r \neq
\pm 1$, then $f$ sends $(-r,0,r)$ to $(-\frac1r, 1, \frac1r)$,
which is not an AP because $p > 2$.~However, $f$ sends $(-1,0,1)$
to $(-1, 1, 0)$, which is an AP if and only if $p = 3$.
\item Next, consider APs of the form $(0, r, 2r)$. If $\{r, 2r\}
\cap \{1\} = \varnothing$, then $f$ sends $(0,r,2r)$ to $(1, \frac
1r, \frac{1}{2r})$, which is an AP if and only if $\frac2r = 1
+ \frac{1}{2r}$, and this happens if and only if $p > 3$ and $r
= \frac32$. If $r =1$, then $f$ sends $(0,r,2r)$ to $(1, 0, \frac
12)$, which is an AP if and only if $p = 3$. If $2r =1$, then $f$
sends $(0,r,2r)$ to $(1, 2, 0)$, which is again an AP if and only
if $p = 3$.
\item[] \hspace{-.33in}
We may now restrict our attention to APs containing $1$ but not $0$.
\item Consider APs of the form $(1-r, 1, 1 + r)$, where $r \neq
\pm 1$. The permutation $f$ sends this AP to $\left(\frac{1}{1-r},
0, \frac{1}{1+r}\right)$, which is not an AP because $p > 2$.
\item  Finally, consider APs of the form
$(1, 1+r, 1+2r)$ with $r \notin \{-1, -\frac12\}$.
Then $f$ sends $(1,1+r,1+2r)$ to
$\left(0, \frac{1}{1+r}, \frac{1}{1+2r}\right)$,
which is an AP if and only if $\frac{2}{1+r} = \frac{1}{1+2r} + 0$,
and this happens if and only if $p > 3$ and $r = - \frac13$.\qedhere
\end{enumerate}
\end{proof}
\begin{remark}
Because $f$\/ is an involution, it also acts as an involution on the
set of APs that are {\em not}\/ \mbox{destroyed by~$f$.}
\end{remark}
Our strategy is to modify $f$\/ by composing it with a permutation $\tau$
that is a simple transposition if $p=3$ and a product of two transpositions
$\tau_1, \tau_2$ if $p>3$, with each transposition moving exactly one term
in each of the APs not destroyed by~$f$.  The resulting permutation $f'$
then destroys those APs but may restore others.  However, if we choose
$\tau$ at random then the expected number of restored APs is $O(1)$,
so once $q$ is at all large there should be some choices of $\tau$
for which no AP is restored and thus $f'$ destroys all APs.
We will prove this by counting how many $\tau$ or $\tau_i$
move a given AP term and introduce no new APs, and showing that
the count is positive.  Because our $f$\/ is given by an algebraic rule,
each of the needed enumerations reduces to estimating the number of points
on certain algebraic curves over~$\FF_q$.  The estimates suffice
with few enough small exceptions (each with $p>3$)
that we can dispose of each remaining~$q$ computationally.
In most cases we find some $\tau_i$ that works
(even though the estimate was not strong enough to guarantee its existence).
In the remaining cases, $q$ is prime and small enough that an
AP-destroying permutation of $\ZZ/p\ZZ$ was already exhibited
by Hegarty in~\cite{hegori}; we also construct such permutations by
starting from the $\tau_i$ that come closest to destroying all APs
and then composing with further transpositions
until the number of surviving APs drops to zero.

In what follows, we consider the cases of $p = 3$ and $p > 3$ separately,
because we saw in the proof of Lemma~\ref{lem:almostwreck} that
the APs not destroyed by~$f$\/ are different in each case.

\subsection{The Case $p = 3$}\label{sec:3case}

By hypothesis $k > 1$, so $q = 3^k > 3$; hence there exists
$y$ in $\FF_q \setminus \{0, \pm 1\}$.  Consider the permutation $f'$
obtained by switching the images of $-1$ and $y$ under~$f$\/
(i.e.~$f'(-1) = 1/y$, $f'(y) = -1$, and $f'(x) = f(x)$ otherwise).
The AP $(-1,0,1)$ is clearly destroyed by $f'$, because $y \neq -1$.
Thus, by Lemma~\ref{lem:almostwreck}, if an AP is not destroyed by~$f'$,
it must either \textsf{(A)} contain $-1$ but not $y$ or
\textsf{(B)} contain $y$.
We treat the cases \textsf{(A)} and \textsf{(B)} separately as follows.
\begin{enumerate}
\item[{\sf (A)}] If $\{-1+r, -1+2r\} \cap \{0,\pm 1, y\} = \varnothing$,
then $f'$ sends the AP $(-1,-1+r,-1+2r)$ to $\left(\frac{1}{y},\frac{1}{-1+r},\frac{1}{-1+2r}\right)$,
which is an AP when
$$
\frac{2}{-1+r} = \frac{1}{y} + \frac{1}{-1+2r} \ \Longrightarrow \ r^2 = y+1.
$$
Otherwise, if $\{-1+r, -1+2r\} \cap \{0,\pm 1\} \neq \varnothing$, then we observe that $f'$ destroys the AP \mbox{$(-1,-1+r,-1+2r)$} if and only if $f'$ destroys the AP $(-1,0,1)$, but this holds by the construction of $f'$.
(We intersected with $\{0,\pm 1\}$, not $\{0,\pm 1,y\}$,
because $y \in \{-1+r, -1+2r\}$ belongs in case~\textsf{(B)}.)
\item[{\sf (B)}] If $\{y+r,y+2r\} \cap \{0,\pm 1,y\} = \varnothing$,
then $f'$ sends the AP $(y,y+r,y+2r)$ to
$\left(-1,\frac{1}{y+r},\frac{1}{y+2r}\right)$, which is an AP when
$$
\frac{2}{y+r} = -1 + \frac{1}{y+2r} \ \Longrightarrow \ r^2 = y^2 + y.
$$
Also, one readily checks that if
$\{y+r,y+2r\} \cap \{0,\pm 1\} \neq \varnothing$,
the AP $(y,y+r,y+2r)$ is destroyed by $f'$. To verify this claim, note that there are $6$ cases to consider, depending on which of $y+r$ and $y+2r$ belongs to the set $\{0, \pm 1\}$. For the sake of clarity, we shall work out the case where $y+r = 0$; the remaining five cases may be handled analogously. If $y + r = 0$, then because $y \not\in \{0, \pm 1\}$, we have that $f'$ sends the AP $(y, y+r, y+2r) = (-r, 0, r)$ to $(-1, 1, \tfrac{1}{r})$, which is an AP if and only if $\tfrac{1}{r} = 3$, but this is of course impossible modulo $3$. Thus, $f'$ destroys the AP $(y, y+r, y+2r)$ when $y + r = 0$.
\end{enumerate}

Now, let $\chi$ denote the Legendre symbol over $\FF_q$. It follows
from the above case analysis that if $y \not\in \{ 0, \pm 1 \}$ is chosen
so that
$$
\big(1-\chi(y+1)\big) \cdot \big(1-\chi(y(y+1))\big) > 0,
$$
then $f'$ destroys all APs in $\FF_q$.
Such a $y$ exists if and only if the sum
\begin{equation}\label{eq:sum3}
A_q(y) \defeq \sum_{y \in \FF_q \setminus \{0, \pm 1\}} \big(1-\chi(y+1)\big)
\cdot \big(1-\chi(y(y+1))\big)
\end{equation}
is positive.
To compute $A_q(y)$, we use the following well-known elementary formula:
\begin{lemma}\label{thm:weil1}
Let $\FF_q$ be a finite field of odd characteristic,
and let $\chi$ be the Legendre symbol on~$\FF_q$.
If $q$ is odd and $g \in \FF_q[x]$ is a polynomial
of degree at most 2 such that $g \neq c \cdot h^2$ for any $c \in
\FF_q$ and $h \in \FF_q[x]$, then
\begin{equation}\label{thm2}
\sum_{y \in \FF_q} \chi(g(y)) = -\chi(a),
\end{equation}
where $a$ is the coefficient of the degree-$2$ term in $g$.
\end{lemma}
Now, taking the sum in~\eqref{eq:sum3} over all of $\FF_q$ and
applying the result of Lemma~\ref{thm:weil1}, \mbox{we find that}
$$
\sum_{y \in \FF_q} \big(1+\chi(y) - \chi(y+1) - \chi(y(y+1))\big) = q+1.
$$
Because
$$
\sum_{y \in \{0, \pm 1\}} \big(1+\chi(y) - \chi(y+1) - \chi(y(y+1))\big) = 3 - \chi(-1) \leq 4,
$$
we conclude that $A_q(y) > 0$ provided $q + 1 > 4$,
which happens once $k \geq 2$.
Thus, there exists $y \in \FF_q \setminus \{0, \pm 1\}$
such that $f'$ destroys all APs in $\FF_q$ for $q = 3^k$ and $k \geq 2$.

\begin{remark}
Alternatively, note that we can handle the case $p = 3$ by simply exhibiting AP-destroying permutations of $\FF_9$ and $\FF_{27}$, for it would then follow by Lemma~\ref{lem:quot} that there is an AP-destroying permutation of $\FF_{3^k}$ for each $k > 1$. Making the identification $\FF_9 \simeq \FF_3[\alpha]/(\alpha^2 + 2\alpha + 2)$, one readily checks that the permutation $f'$ of $\FF_9$ obtained by taking $y = \alpha + 1$ destroys all APs in $\FF_9$. Similarly, making the identification $\FF_{27} \simeq \FF_3[\beta]/(\beta^3 + 2\beta + 1)$, one readily checks that the permutation $f'$ of $\FF_{27}$ obtained by taking $y = \beta^2$ destroys all APs in $\FF_{27}$. Nonetheless, this \emph{ad hoc} argument does not readily generalize to primes $p > 3$, while the proof provided prior to the present remark extends quite naturally to primes $p > 3$, as we demonstrate in Sections~\ref{simsec} and~\ref{musthavelabel}.
\end{remark}

\subsection{Destroying $(0, \frac32 ,3)$ in the Case $p > 3$}\label{simsec}

This case takes more work than the case $p=3$, but the strategy is similar.
We begin by constructing a permutation $f'$ of $\FF_q$
that destroys all but one AP.
Take $y \in \FF_q \setminus \{0, \frac13, \frac 23,1, \frac32,3\}$
(note that this already requires $q>5$), and let $f'$ be the permutation
obtained by switching the images of $3$ and~$y$ under~$f$\/; that is,
$f'(3) = \frac1y$, $f'(y) = \frac13$, and $f'(x) = f(x)$ otherwise.
The AP $(0,\frac32,3)$ is clearly destroyed by $f'$, because $y
\neq 3$. Thus, by Lemma~\ref{lem:almostwreck}, if an AP other than $(\frac 13, \frac 23, 1)$ is not
destroyed by $f'$, it must either \textsf{(A)} contain $3$ but
not $y$ or \textsf{(B)} contain $y$. The cases \textsf{(A)} and
\textsf{(B)} each have two subcases depending on the position of
$3$ or~$y$ in the AP; we study each of these subcases separately as follows.
\begin{enumerate}
\item[{\sf (A)}]
\begin{enumerate}
\item If $\{3+r, 3+2r\} \cap \{0,1,3, y\} = \varnothing$, then
$f'$ sends the AP $(3,3+r,3+2r)$ to
$\left(\frac{1}{y},\frac{1}{3+r},\frac{1}{3+2r}\right)$, which is an AP when
$$
\frac{2}{3+r} = \frac{1}{y} + \frac{1}{3+2r}
\ \Longrightarrow \
2r^2 + (9-3y)r + (9-3y) = 0.
$$
If $r = -3$ and $3 + 2r = -3 \neq y$, then $f'$ does not destroy
the AP $(3,3+r,3+2r) = (3,0,-3)$ when $y = \frac 37$.
One readily checks that $f'$ destroys all other APs of the form
$(3,3+r,3+2r)$ that do not contain $y$ and satisfy $\{3+r, 3+2r\}
\cap \{0,1,3\} \neq \varnothing$. (Note that there are $3$ cases left to consider, making a total of four cases in all, according as $3+r$ or $3+2r$ belongs to $\{0,1\}$.)
\item If $\{3-r, 3+r\} \cap \{0,1,3, y\} = \varnothing$,
then $f'$ sends the AP $(3-r,3,3+r)$ to
$\left(\frac{1}{3-r},\frac{1}{y},\frac{1}{3+r}\right)$,
which is an AP when
$$
\frac{2}{y} = \frac{1}{3-r} + \frac{1}{3+r}
\ \Longrightarrow \ r^2
= 9-3y.
$$
If $r = 3$ and $3 + r = 6 \neq y$, then $f'$ does not destroy the
AP $(3-r,3,3+r) = (0,3,6)$ when $y
= \frac{12}{7}$.
If $r = 2$ and $3+r = 5 \neq y$, then $f'$ does not destroy the
AP $(3-r, 3, 3+r)  = (1,3,5)$ when $y = 10$.
There are no cases that remain to be considered for the AP $(3-r, 3, 3+r)$.
\end{enumerate}
\item[{\sf (B)}]
\begin{enumerate}
\item If $\{y+r,y+2r\} \cap \{0,1,3,y\} = \varnothing$, then $f'$
sends the AP $(y,y+r,y+2r)$ to $\left(\frac{1}{3},\frac{1}{y+r},\frac{1}{y+2r}\right)$,
which is an AP when
$$
\frac{2}{y+r} = \frac{1}{3} + \frac{1}{y+2r}
\ \Longrightarrow\
2r^2 + (3y-9)r + (y^2 - 3y) = 0.
$$
If $y+r = 0$ and $y+2r \not\in \{0,1,3\}$, then $f'$ does not destroy
the AP $(y,y+r,y+2r) = (y,0,-y)$ when
$5y = -3$.
If $y+r = 1$ and $y+2r \not\in \{0,1,3\}$, then $f'$ does not destroy
the AP $(y,y+r,y+2r) = (y,1,2-y)$ when $y = 5$.
If $y+r = 3$ and $y+2r \not\in \{0,1,3\}$, then $f'$ does not destroy
the AP $(y,y+r,y+2r) = (y,3,6-y)$ when $y = 12$.
If $y+2r = 1$ and $y+r \not\in \{0,1,3\}$, then $f'$ does not destroy
the AP $(y,y+r,y+2r) = \left(y,\frac{y+1}{2},1\right)$ when $y
= 11$.
One readily checks that $f'$ destroys all other APs of the form
$(y,y+r,y+2r)$ that satisfy $\{y+r, y+2r\}
\cap \{0,1,3\} \neq \varnothing$. (Note that the remaining cases here are when $y + 2r \in \{0,3\}$ and $y+r \not\in \{0,1,3\}$ and when both $y+r$ and $y+2r$ are in $\{0,1,3\}$.)
\item If $\{y-r,y+r\} \cap \{0,1,3,y\} = \varnothing$, then $f'$
sends the AP $(y-r,y,y+r)$ to $\left(\frac{1}{y-r},\frac{1}{3},\frac{1}{y+r}\right)$,
which is an AP when
$$
\frac{2}{3} = \frac{1}{y-r} + \frac{1}{y+r} \ \Longrightarrow\ r^2 = y^2 - 3y.
$$
If $y-r = 0$ and $y+r \not\in \{0,1,3\}$, then $f'$ does not destroy
the AP $(y-r,y,y+r) = (0,y,2y)$ when $y = - \frac32$.
If $y-r = 1$ and $y+r \not\in \{0,1,3\}$, then $f'$ does not destroy
the AP $(y-r,y,y+r) = (1,y,2y-1)$ when $y = \frac54$.
If $y-r = 3$ and $y+r \not\in \{0,1,3\}$, then $f'$ does not destroy
the AP $(y-r,y,y+r) = (3,y,2y-3)$ when $y= \frac34$. The only remaining
AP of the form $(y-r,y,y+r)$ that satisfies
$\{y-r, y+r\} \cap \{0,1,3\} \neq \varnothing$ is given by $(y-r, y, y+r) = (1,2,3)$, but $f'$ evidently destroys this AP.
\end{enumerate}
\end{enumerate}
Now, taking $\chi$ to be the Legendre symbol over $\FF_q$ as before,
it follows from the above case analysis that if
$$y \not\in S \defeq \left\{-\frac{3}{2},-\frac{3}{5},0, \frac{1}{3},\frac{3}{7},\frac{2}{3},\frac{3}{4},1,\frac{5}{4},\frac{12}{7},\frac{3}{2},3,5,10,11,12\right\}$$
is chosen so that
\begin{equation}\label{ycond}
\big(1-\chi((3-y)(3-9y))\big) \cdot \big(1-\chi(3(3-y))\big) \cdot
\big(1-\chi((3-y)(27-y))\big) \cdot \big(1-\chi(-y(3-y))\big) >
0,
\end{equation}
then $f'$ destroys all APs other than the AP $(\frac13, \frac
23,1)$ in $\FF_q$. Expanding the product on the left-hand side (LHS)
of~\eqref{ycond} under the assumption that $y \neq 3$ (so that
$\chi((3-y)^2) = 1$), we obtain a lengthy expression that we denote
by $B_q(y)$ for the sake of readability:
 \scriptsize
 \begin{align*}
B_q(y) & \defeq \big[1-\chi((3-y)(\tfrac{1}{3}-y)) - \chi(9-3y)
- \chi((3-y)(27-y))-\chi(y(y-3)) +\chi(1-3y) + \label{fateqn}\\
& \hphantom{==}\chi((27-y)(\tfrac{1}{3}-y))  + \chi(y(y-\tfrac{1}{3}))
+ \chi(3(27-y)) +\chi(-3y) + \chi(y(y-27))  - \chi((3-y)(27-y)(1-3y))
 -  \nonumber\\
& \hphantom{==}\chi(y(3-y)(3y-1)) -\chi(y(3-y)(27-y)(y-\tfrac{1}{3}))
- \chi(3y(3-y)(y-27)) + \chi(y(27-y)(3y-1))\big].\nonumber
\end{align*}
\normalsize
Clearly, there exists $y \in \FF_q \setminus S$ satisfying~\eqref{ycond}
if and only if
\begin{equation}\label{opossum}
\sum_{y \in \FF_q \setminus S} B_q(y) > 0.
\end{equation}
To estimate the LHS of~\eqref{opossum}, we first estimate
$\sum_{y \in \FF_q} B_q(y)$, for which must invoke not only
Lemma~\ref{thm:weil1} but also the Hasse bound~(see~\cite{borup} for the original paper and Corollary 1.4 of~\cite{silversurfer} for a more modern reference):
\begin{theorem}\label{thm:weil2}
{\rm [Hasse]}
Let $\FF_q$ be a finite field of odd characteristic,
and let $\chi$ be the Legendre symbol on $\FF_q$.
If $g \in \FF_q[x]$ is a polynomial of degree $3$ or
$4$ such that $g \neq c \cdot h^2$ for any $c \in \FF_q$ and $h
\in \FF_q[x]$, then
\begin{equation}\label{thm1}
\Bigl|\chi(a) + \sum_{y \in \FF_q} \chi(g(y))\Bigr| \leq 2\sqrt{q},
\end{equation}
where $a$ is the coefficient of the degree-$4$ term in $g$.
\end{theorem}
The constant term $1$ in $B_q(y)$ yields a contribution of $q$ to the sum. Each of the other terms in $B_q(y)$ is of the form $\chi(g(y))$ for some
polynomial $g \in \FF_q[x]$. Applying Lemma~\ref{thm:weil1} to
the terms of the form $\chi(g(y))$ where $g$ has degree at most~$2$,
we obtain a total contribution of $0$ from such terms. Then,
applying Theorem~\ref{thm:weil2} to the remaining terms, which
are of the from $\chi(g(y))$ where $g$ has degree $3$ or $4$,
yields
\begin{equation}\label{est1}
\sum_{y \in \FF_q} B_q(y) \geq q- 10\sqrt{q} - 1.
\end{equation}
Estimating $\sum_{y \in S} B_q(y)$ by using the trivial bound $|\chi(g(y))|
\leq 1$ for each $y \in S$ unless $g(y) = 0$ (which can occur when
$y \in \{0,\frac13,3\}$), we find that
\begin{equation}\label{est2}
\sum_{y \in S} B_q(y) \leq 13 \cdot 16 + 0 + 8 + 8 = 224,
\end{equation}
where the term $13 \cdot 16$ bounds the contributions of $y \in S \setminus \{0,\tfrac 13, 3\}$, the term $0$ is the contribution of $y = 0$, and the terms $8 + 8$ bound the contributions of $y \in \{\tfrac 13 , 3\}$. Combining the estimates~\eqref{est1} and~\eqref{est2}, we deduce
that~\eqref{opossum} holds if $q - 10\sqrt{q} \geq 225$, which happens
when $q \geq 434$. Thus, there exists $y \in \FF_q \setminus S$
such that $f'$ destroys all APs other than the AP $(\frac13,\frac
23,1)$ in $\FF_q$ for $q \geq 434$ a prime power.

\subsection{Destroying $(\frac13, \frac23,1)$ in the Case $p > 3$}\label{musthavelabel}

We now perform an analogous maneuver to construct a permutation
that destroys \emph{all} APs in $\FF_q$ for sufficiently large
$q$. With $q \geq 434$, take $y \in \FF_q \setminus S$ so that
the permutation $f'$ destroys all APs other than the AP
$(\frac 13, \frac23,1)$, and for
$z \in \FF_q \setminus \{0, \frac13, \frac 23,1, \frac 32,3,y\}$,
consider the permutation $f''$ obtained by switching the images of
$\frac13$ and $z$ under $f'$ (i.e.~$f''(\frac13) = \frac1z$,
$f''(z) = 3$, and $f''(x) = f'(x)$ otherwise).
The AP $(\frac13,\frac23,1)$ is clearly destroyed by $f''$, as
$z \neq \frac13$. Thus, by Lemma~\ref{lem:almostwreck}, if an
AP is not destroyed by $f''$, it must either \textsf{(A)} contain
$\frac13$ but not $z$ or \textsf{(B)} contain $z$. The cases \textsf{(A)}
and \textsf{(B)} each have two subcases depending on the position
of $\frac13$ or $z$ in the AP; we study each of these subcases
separately as follows:
\begin{enumerate}
\item[{\sf (A)}]
\begin{enumerate}
\item If $\{\frac13+r, \frac13+2r\} \cap \{0,\frac13,1,3,y,z\}
= \varnothing$, then $f''$ sends the AP $(\frac13,\frac13+r,\frac
13+2r)$ to $\left(\frac{1}{z},\frac{1}{\frac13+r},\frac{1}{\frac
13+2r}\right)$, which is an AP when
$$
\frac{2}{\frac13+r} = \frac{1}{z} + \frac{1}{\frac13+2r}
\ \Longrightarrow\
2r^2 + (1-3z)r + (\tfrac19 - \tfrac z3) = 0.
$$
We must now deal with the cases where the AP does not contain $z$
but $\{\frac13+r, \frac13+2r\} \cap \{0,\frac13,1,3,y\} \neq
\varnothing$; as the computations are more complicated and numerous
in the present situation, we will not be as explicit as we were
in Section~\ref{simsec}.
If $\frac13 + r = 0$ and $\frac13 + 2r = - \frac13 \neq z$,
then there is at most one value of $z$, call it $a_1$, such that
$f''$ does not destroy the AP $(\frac13, 0, -\frac13)$. (The
value $a_1$, if it exists, is uniquely determined by the particular
field $\FF_q$ with which we are working.)
Similarly, from the cases where $\frac13 + r \in \{1,3,y\}$, there
are at most three additional values of $z$, call them $a_2, a_3, a_4$,
such that $f''$ does not destroy the AP $(\frac13, \frac13+r, \frac13+2r)$.
From the cases where $\frac13 + 2r \in \{0,3,y\}$, there are at most
three values of $z$, call them $a_5, a_6, a_7$, such that
$f''$ does not destroy the AP $(\frac13, \frac13+r, \frac13+2r)$.
 The only remaining
case to consider is when $\tfrac 13 + 2r = 1$, but the corresponding AP $(\tfrac 13, \tfrac 13 + r, \tfrac 13 + 2r) = (\tfrac 13, \tfrac 23, 1)$ is, as mentioned before, destroyed by $f'$.
\item If $\{\frac13 -r, \frac13 +r\} \cap \{0,\frac13,1,3,y,z\} = \varnothing$,
then $f''$ sends the AP $(\frac 13 -r,\frac 13,\frac 13+r)$ to
$\left(\frac{1}{\frac 13 -r},\frac{1}{z},\frac{1}{\frac 13 +r}\right)$,
which is an AP when
$$
\frac2z = \frac{1}{\frac13 - r} + \frac{1}{\frac 13 +r} \ \Longrightarrow\
r^2 = \frac19 - \frac{z}{3}.
$$
From the cases where $\frac 13 - r \in \{0,1,3,y\}$, there are
at most four values of $z$, call them $a_8, a_9, a_{10}, a_{11}$,
such that $f''$ does not destroy the AP $(\frac 13 -r,\frac 13,\frac13+r)$. By symmetry, we have also taken care of the cases where
$\frac 13 + r \in \{0,1,3,y\}$. There are no cases that remain to be considered for the AP $(\frac 13 -r,\frac 13,\frac 13+r)$.
\end{enumerate}
\item[{\sf (B)}]
\begin{enumerate}
\item If $\{z+r,z+2r\} \cap \{0,\frac 13,1,3,y,z\} = \varnothing$,
then the AP $(z,z+r,z+2r)$ is sent under $f''$ to $\left(3,\frac{1}{z+r},\frac{1}{z+2r}\right)$,
which is an AP when
$$
\frac{2}{z+r} = 3 + \frac{1}{z+2r} \ \Longrightarrow\
6r^2 + (9z-3)r + (3z^2 - z) = 0.
$$
The cases where $\{z+r,z+2r\} \cap \{0,\frac 13,1,3,y\} \neq \varnothing$
give at most $29$ values of $z$, call them $a_{12}, \dots, a_{40}$,
such that $f''$ does not destroy the AP $(z,z+r,z+2r)$. To see why, observe that there are at most five values of $z$ arising from the possibility that $z+r \in \{0, \frac 13, 1, 3, y\} \not\ni z+ 2r$, because for each value of $z+r \neq \frac 13$ we obtain a linear equation in $z$, and when $z + r = \frac 13$, we obtain a quadratic equation one of whose solutions is $z = \frac 13$ and must therefore be discarded. Similarly, we obtain at most five values of $z$ arising from the possibility that $z+2r \in \{0, \frac 13, 1, 3, y\} \not\ni z+ r$. In the cases where $z+r, z+2r \in \{0, \frac 13, 1, 3, y\}$, we can solve for $z$ immediately (without even imposing the condition that $f''$ destroys the AP $(z, z+r, z+2r)$); we obtain at most $20$ values of $z$, one corresponding to each of the $5 \cdot 4 = 20$ different ordered pairs of distinct elements of $\{0, \frac 13, 1, 3, y\}$. Nevertheless, it is clear that we should discard the case where $z+r= \frac 13$ and $z+2r = 0$, which would imply that $z = \frac 23$, contradicting our restriction on the value of $z$. We therefore end up with at most $5 + 5 + 20 - 1 = 29$ values of $z$. There are no cases that remain to be considered for the AP $(z,z+r,z+2r)$.
\item If $\{z-r,z+r\} \cap \{0,\frac 13,1,3,y,z\} = \varnothing$,
then the AP $(z-r,z,z+r)$ is sent under $f''$ to $\left(\frac{1}{z-r},3,\frac{1}{z+r}\right)$,
which is an AP when
$$
6 = \frac{1}{z-r} + \frac{1}{z+r} \ \Longrightarrow\ r^2 = z^2 - \frac{z}{3}.
$$
The cases where $\{z-r,z+r\} \cap \{0,\frac 13,1,3,y\} \neq \varnothing$
give at most $14$ values of $z$, call them $a_{41}, \dots, a_{54}$,
such that $f''$ does not destroy the AP $(z-r,z,z+r)$. To see why, observe that there are at most five values of $z$ arising from the possibility that $z-r \in \{0, \frac 13, 1, 3, y\} \not\ni z+ r$, because for each value of $z-r \neq \frac 13$ we obtain a linear equation in $z$, and when $z - r = \frac 13$, we obtain a quadratic equation one of whose solutions is $z = \frac 13$ and must therefore be discarded. By symmetry, we have also taken care of the possibility that $z+r \in \{0, \frac 13, 1, 3, y\} \not\ni z- r$. In the cases where $z-r, z+r \in \{0, \frac 13, 1, 3, y\}$, we can solve for $z$ immediately (without even imposing the condition that $f''$ destroys the AP $(z-r, z, z+r)$); we obtain at most $10$ values of $z$, one corresponding to each of the ${5 \choose 2} = 10$ different pairs of distinct elements of $\{0, \frac 13, 1, 3, y\}$. Nevertheless, it is clear that we should discard the case where $z-r= 0$ and $z+r = 3$, which would imply that $z = \frac 32$, contradicting our restriction on the value of $z$. We therefore end up with at most $5 + 10 - 1 = 14$ values of $z$. There are no cases that remain to be considered for the AP $(z-r, z, z+r)$.
\end{enumerate}
\end{enumerate}
As in Section~\ref{simsec}, we can use the above case analysis to write down a condition on when $f''$ destroys all APs in $\FF_q$. Indeed, if
$$
z \not\in S' \defeq
   \left\{a_i : 1 \leq i \leq 54\right\}
 \cup
  \left\{0, \frac 13, \frac 23,1, \frac 32,3,y\right\}
$$
is chosen so that we have
\begin{equation}\label{neweqs}
\big(1-\chi((1-27z)(1-3z))\big) \cdot \big(1-\chi(1-3z)\big) \cdot
\big(1-\chi((1-3z)(1 - \tfrac z3 ))\big) \cdot \big(1-\chi(-3z(1-3z))\big)
> 0,
\end{equation}
then $f''$ destroys all APs in $\FF_q$. Let $w \defeq \tfrac{1}{z}$ for $z \neq 0$, and define $S'' \defeq \{\tfrac{1}{x} : x \in S' \setminus \{0\}\} \cup \{0\}$. Then, rewriting the above condition in terms of $w$ and $S''$, we obtain the following ``new condition'': if $w \not\in S''$ is chosen so that
\small
\begin{equation}\label{wcond}
\big(1-\chi((3-w)(3-9w))\big) \cdot \big(1-\chi(3(3-w))\big) \cdot
\big(1-\chi((3-w)(27-w))\big) \cdot \big(1-\chi(-w(3-w))\big) >
0,
\end{equation}
\normalsize
then $f''$ destroys all APs in $\FF_q$. But upon making the replacements $y \rightsquigarrow w$ and $S \rightsquigarrow S''$, one readily observes that this ``new condition'' is the same as the analogous condition obtained in Section~\ref{simsec}, namely~\eqref{ycond}.
Therefore, there exists $w \in \FF_q \setminus S''$ satisfying~\eqref{neweqs}
if and only if
\begin{equation}\label{opossum2}
\sum_{w \in \FF_q \setminus S''} B_q(w) > 0.
\end{equation}
To estimate the LHS of~\eqref{opossum2}, we first recall the bound~\eqref{est1}:
\begin{equation}\label{est12}
\sum_{w \in \FF_q} B_q(w) \geq q - 10\sqrt{q}-1.
\end{equation}
Next, estimating $\sum_{w \in S''} B_q(w)$ by using the trivial bound
$|\chi(g(w))| \leq 1$ for each $w \in S''$ unless $g(w) = 0$ (which
can occur when $w \in \{0,\frac 13,3\}$), we find that
\begin{equation}\label{est22}
\sum_{w \in S''} B_q(w) \leq 58 \cdot 16 + 0 + 8 + 8 = 944,
\end{equation}
where the term $58 \cdot 16$ bounds the contributions of $w \in S'' \setminus \{0,\tfrac 13, 3\}$, the term $0$ is the contribution of $w = 0$, and the terms $8 + 8$ bound the contributions of $w \in \{\tfrac 13 , 3\}$. Combining the estimates~\eqref{est12} and~\eqref{est22}, we deduce
that~\eqref{opossum2} holds if $q - 10\sqrt{q} \geq 945$, which
happens when $q \geq 1307$. Thus, there exists $w \in \FF_q \setminus
S''$ such that $f''$ destroys all APs in $\FF_q$ for $q \geq 1307$
a prime power.

\subsection{Remaining Cases}

We have now shown that there exists an AP-destroying permutation
of $(\ZZ/p\ZZ)^k$ if $p = 3$ and $k \geq 2$ and if $p^k \geq 1307$.
To complete the proof of Theorem~\ref{thm:main}, it remains to
check the finitely many remaining cases, and we do this by resorting
to a computer program.  By Lemma~\ref{lem:quot}, it
suffices to check that $(\ZZ/p\ZZ)^k$ has an AP-destroying permutation
for the following cases:
$(p,k) \in\{(p,1) : 7 < p < 1307 \text{ is prime}\} \cup \{(5,2),(5,3),(7,2),(7,3)\}$.
In each of these cases other than $p = q = 11, 13, 29, 31$, our code relies
on the construction used in the argument of Section~\ref{sec:proof}.
Indeed, for these cases, the idea of modifying the values of $f(3)$
and $f(\frac 13)$ actually works to yield an AP-destroying permutation
of $(\ZZ/p\ZZ)^k$.
For $q \in \{11, 13, 29, 31\}$, explicit AP-destroying permutations of
$\ZZ/p\ZZ$ were constructed by Hegarty in~\cite{hegori}. The code
required to check these cases, as well as a database listing the explicit AP-destroying permutations for all of the above exceptional cases, may be obtained by downloading the source files from the following website: \url{https://arxiv.org/format/1601.07541v3}.

\section*{Acknowledgments}

\noindent The first author is supported by NSF grants DMS-1100511 and DMS-1502161.
The second author acknowledges Joseph Gallian for
suggesting this problem during the 2015 Duluth Mathematics REU,
and was supported by NSF grant DMS-1358695 and NSA grant H98230-13-1-0273.
He also thanks Amol Aggarwal, Levent Alpoge, Peter Hegarty, Mitchell Lee,
and Julian Sahasrabudhe for useful conversations and suggestions. Both authors thank
Aleksey Fedotov for helpful comments, as well as the anonymous referee, who made many insightful comments, suggestions, and corrections \mbox{regarding various important aspects of the article.}

\noindent

\bibliographystyle{amsxport}
\bibliography{bibfile}

{\bigskip\medskip{\footnotesize%
  \textsc{Noam D. Elkies. Department of Mathematics, Harvard University, \mbox{Cambridge, MA 02138}} \par
  \textit{E-mail address}: \url{elkies@math.harvard.edu} \par
  \addvspace{\medskipamount}
  \textsc{Ashvin A. Swaminathan. Department of Mathematics, Harvard University, \mbox{Cambridge, MA 02138}} \par
\textit{E-mail address}: \url{aaswaminathan@college.harvard.edu}
}}

\end{document}